\documentclass[11pt]{amsart}
\usepackage{tikz}
\usetikzlibrary{positioning,arrows.meta}
\usepackage{amsmath,amssymb,amsfonts}
\usepackage{amsthm}
\usepackage{graphicx}
\usepackage{bm}
\usepackage{cite}
\usepackage{geometry}
\usepackage{booktabs}
\usepackage{graphicx}
\usepackage{amsmath}
\usepackage{amsfonts}
\usepackage{amsthm}
\usepackage{amssymb,latexsym}
\usepackage{fixmath}
\usepackage{mathrsfs,amsbsy}
\usepackage{graphicx}
\usepackage{caption}
\usepackage{subcaption}		
\usepackage{bm}
\usepackage{caption}
\usepackage{makecell}
\usepackage{xr}
\usepackage{mathrsfs,amsbsy}
\usepackage{dsfont}
\usepackage{enumerate}
\usepackage{eucal}
\usepackage{color} %red, green, blue, yellow, cyan, magenta, black, white
\definecolor{mygreen}{RGB}{28,172,0} % color values Red, Green, Blue
\definecolor{mylilas}{RGB}{170,55,241}
\usepackage{xcolor}

\definecolor{darkblue}{rgb}{0,0,0.6}

\newcommand{\rev}[1]{#1}
\theoremstyle{plain}
\newtheorem{theorem}{Theorem}[section]

\newtheorem{proposition}[theorem]{Proposition}

\theoremstyle{definition}

\theoremstyle{remark}
\newtheorem{remark}[theorem]{Remark}

\usepackage{xcolor}
\usepackage{adjustbox}
\definecolor{mygreen}{RGB}{28,172,0} % color values Red, Green, Blue
\definecolor{mylilas}{RGB}{170,55,241}

\title{A scalar interface reduction for nonlinear interface problems}
\author{ So-Hsiang Chou}			
\thanks{
 Department of Mathematics and Statistics, Bowling Green
State university, Bowling Green, OH, 43403-0221.
{{\tt email:chou@bgsu.edu}}
}
\date{\today}	
\begin{document}
\begin{abstract}

We study finite element approximations of elliptic and parabolic interface problems with discontinuous coefficients and nonlinear jump conditions. We introduce a scalar interface reduction in which the solution is decomposed into a continuous component and a unit–jump response mode. This representation isolates the interface nonlinearity into a single scalar variable while the bulk problem remains linear.

From this perspective, the nonlinear interface condition is reduced to a scalar nonlinear equation, which may be interpreted as a nonlinear Schur complement associated with the interface degree of freedom. The resulting formulation leads to a simple computational procedure consisting of linear solves combined with a low-dimensional nonlinear update.

Numerical results for representative elliptic and parabolic problems confirm second-order accuracy for interface quantities and demonstrate the effectiveness of the proposed approach.
\end{abstract}
\maketitle
%{\bf \small Key Words.}\keywords{\small{ finite difference method, nonlinear parabolic problems, ion sensors, heat equation.}
%\textbf{Mathematics Subjects Classification}: 65N30, 65N35, 35J05

\section{Introduction}

Interface problems with discontinuous coefficients arise in many
applications, including heat conduction in composite materials and
diffusion in heterogeneous media. In such problems the governing
equation holds separately in subdomains, while transmission conditions
couple the traces of the solution across the interface.

In this paper we consider interface problems with nonlinear jump
conditions and develop a finite element formulation based on a scalar
interface reduction. The key observation is that the solution may be
represented as
\[
u=u_0+s\,u_1,
\]
where $u_0$ is continuous across the interface and $u_1$ is a
unit--jump response mode determined solely by the differential
operator. The scalar parameter $s$ represents the jump amplitude and
contains the entire effect of the nonlinear transmission law.

Substituting this decomposition into the interface condition reduces
the original nonlinear interface problem to a scalar nonlinear
equation for $s$. The bulk components $u_0$ and $u_1$ are obtained
from linear problems independent of the nonlinear jump law.
Consequently, the nonlinear interface coupling is represented by a
single scalar variable, while the finite element computation remains
linear and standard.

Numerical methods for interface problems have been widely studied.
Among them are immersed interface and hybrid difference methods
\cite{leveque1994immersed,li2006immersed,jeon2022immersed},
immersed and interface finite element methods
\cite{jo2018enriched,lin2015,HouWu1997},
and extended finite element approaches
\cite{babuvska1997partition,fries2010extended}.
Finite difference treatments of nonlinear interface jump conditions
have also been investigated in earlier work. In contrast to these approaches,
the present formulation isolates the
nonlinear interface law into a single scalar variable while the bulk
finite element computation remains linear and standard. Consequently,
the nonlinear component of the computation is reduced to a scalar
equation whose dimension is independent of the mesh size.

The principal contribution of this work is the identification of a
low-dimensional interface reduction for nonlinear transmission
problems. The decomposition separates the bulk response of the
differential operator from the nonlinear interface law and leads to a
computational procedure consisting of linear finite element solves
combined with a scalar nonlinear update. At the discrete level, the
resulting scalar equation may be interpreted as a nonlinear Schur
complement associated with the interface degree of freedom.

Although the presentation is restricted to one-dimensional problems
for clarity, the underlying reduction reflects a structural property
of the interface problem and extends naturally to higher-dimensional
settings; see
\cite{chou2026lifting,Chou2026FluxRecovery}.

The remainder of the paper is organized as follows.
Section~2 presents the interface model, the scalar decomposition, the
finite element discretization, the error decomposition, the algebraic
interpretation, the parabolic extension, and the numerical experiments.
Section~3 contains concluding remarks and discusses possible extensions
of the scalar interface reduction framework.

\section{Interface problem and solution decomposition}

We begin with a time-dependent interface problem to illustrate the
generality of the proposed approach. Let $\Omega \subset \mathbb{R}^d$
be a bounded domain partitioned by an interface $\Gamma$ into two
subdomains $\Omega^{-}$ and $\Omega^{+}$. We consider the parabolic problem
\begin{align}
u_t - \nabla \cdot (\beta(x)\nabla u) &= f(x,t), \qquad x \in \Omega^{-}\cup\Omega^{+}, \ t>0,
\end{align}
subject to appropriate boundary and initial conditions. The diffusion
coefficient $\beta$ is piecewise smooth with distinct values in the two
subdomains.

Across the interface $\Gamma$, we impose continuity of flux together
with a nonlinear jump condition
\begin{align}
[\beta \nabla u \cdot n] &= 0, \\
[u] &= g(u^{+},u^{-}),
\end{align}
where $[\cdot]$ denotes the jump across $\Gamma$. The function
$g$ may represent a variety of nonlinear transmission laws.
Typical examples include polynomial laws such as
\[
[u] = \lambda\, u^{+} u^{-},
\]
or more general nonlinear responses of the form
\[
[u] = g(u^{+},u^{-}),
\]
arising in interface models with nonlinear reactions or contact-type
conditions.

This structure suggests a decomposition of the form
\[
u = u_0 + s\,u_1,
\]
where $u_0$ is continuous across the interface and $u_1$ is a
unit–jump response satisfying $[u_1]=1$ on $\Gamma$ together with
homogeneous boundary conditions. The scalar parameter $s$ represents
the amplitude of the interface jump.

Substituting this decomposition into the governing equations shows that
the bulk problem is independent of the interface nonlinearity and remains
linear, while the nonlinear interface condition reduces to a scalar
nonlinear equation for $s$. In this sense, the full nonlinear coupling
is captured by a single interface degree of freedom.

For clarity of presentation, we focus first on the steady-state
(elliptic) case, where the essential structure of the method can be
fully exposed. We therefore consider the steady-state problem on a one-dimensional
domain $\Omega = (-1,1)$ with an interface point $\alpha \in (-1,1)$:
\begin{align}
- (\beta(x) u')' &= f(x), \qquad x \in (-1,\alpha)\cup(\alpha,1),
\end{align}
where the coefficient $\beta$ is piecewise constant,
\[
\beta(x) =
\begin{cases}
\beta^{-}, & x<\alpha,\\
\beta^{+}, & x>\alpha,
\end{cases}
\]
with $\beta^{-},\beta^{+} > 0$.

The solution satisfies the interface conditions
\begin{align}
[\beta u'] &= 0, \\
[u] &= g(u^{+},u^{-}),
\end{align}
together with appropriate boundary conditions at $x=-1$ and $x=1$, e.g., $u(-1)=\xi,u(1)=\eta$.

This one-dimensional setting provides a minimal framework in which the
interface reduction can be derived explicitly while retaining the full
nonlinear structure of the transmission condition. The nonlinear
interface coupling will be shown to reduce to a scalar equation, while
the remaining components are obtained from linear interface problems.

\subsection{Continuous decomposition}

A key observation is that the interface jump can be isolated through a decomposition of the solution.
We first define the \emph{zero–jump component} $u_0$ as the solution of
\begin{equation}
-(\beta u_{0,x})_x = f(x), \qquad x\in\Omega^- \cup \Omega^+ ,
\end{equation}
with the boundary conditions
\[
u_0(-1)=\xi, \qquad u_0(1)=\eta ,
\]
and interface conditions
\[
[u_0]_\alpha = 0, \qquad [\beta u_{0,x}]_\alpha = 0 .
\]

Next we define the \emph{unit–jump response} $u_1$ as the solution of
\begin{equation}
-(\beta u_{1,x})_x = 0, \qquad x\in\Omega^- \cup \Omega^+ ,
\end{equation}
with homogeneous boundary conditions
\[
u_1(-1)=0, \qquad u_1(1)=0 ,
\]
and interface conditions
\[
[u_1]_\alpha = 1, \qquad [\beta u_{1,x}]_\alpha = 0 .
\]
The homogeneous interface problem with zero boundary conditions admits a
one–dimensional solution space, which is spanned by the unit–jump response $u_1$.
\begin{proposition}
Let $u$ be the solution of the interface problem. Then there exists a scalar
$s$ such that
\[
u = u_0 + s u_1 ,
\qquad s = [u]_\alpha .
\]
\end{proposition}

\begin{proof}
Let $w = u-u_0$.
Subtracting the equations satisfied by $u$ and $u_0$ shows that $w$ satisfies
\[
-(\beta w_x)_x = 0
\]
in each subdomain, together with homogeneous boundary conditions and flux continuity across the interface.
Moreover,
\[
[w]_\alpha = [u]_\alpha .
\]
Since the homogeneous interface problem is linear and its solution space
is one–dimensional, $w$ must be a multiple of the unit–jump solution $u_1$.
Thus $w = s u_1$ with $s=[u]_\alpha$, which yields the result.
\end{proof}

\subsection{Scalar interface equation}

Substituting the decomposition $u = u_0 + s u_1$ into the interface
condition yields a scalar nonlinear equation for the jump amplitude $s$.
\begin{theorem}[Scalar interface reduction]
Let $u_0$ and $u_1$ be defined as above. Then the solution of the
nonlinear interface problem can be written as
\[
u = u_0 + s u_1,
\]
where the scalar $s$ satisfies the nonlinear equation
\[
s = g\big(u_0^+ + s u_1^+,\; u_0^- + s u_1^-\big).
\]
\end{theorem}

\begin{proof}
From the decomposition $u = u_0 + s u_1$, we have
\[
u^\pm = u_0^\pm + s u_1^\pm.
\]
Taking the jump gives
\[
[u] = s.
\]
Substituting into the nonlinear interface condition
\[
[u] = g(u^+,u^-)
\]
yields the stated scalar equation.
\end{proof}
\rev{
\begin{remark}[Existence and uniqueness]
Define
\[
F(s)
=
s-
g\!\left(
u_0^+ + s u_1^+,
u_0^- + s u_1^-
\right).
\]

If $g$ is continuous, existence of a solution follows whenever
$F(s)$ changes sign on a bounded interval.

Furthermore, if
\[
\left|
\frac{\partial g}{\partial u^+}u_1^+
+
\frac{\partial g}{\partial u^-}u_1^-
\right|
<1,
\]
then
\[
F'(s)>0,
\]
and the reduced scalar equation admits a unique solution.

For particular interface laws, such as the quadratic law
\[
[u]=\lambda u^+u^-,
\]
more explicit uniqueness conditions may be obtained.
\end{remark}
}

\begin{proposition}
Suppose the interface law $g(u^+,u^-)$ is a polynomial of degree $k$ in
its arguments. Then the reduced interface equation for $s$ is a
polynomial of degree at most $k$.
\end{proposition}

\begin{proof}
From the representation $u = u_0 + s u_1$, the traces of the solution at
the interface are
\[
u^\pm(\alpha) = u_0^\pm(\alpha) + s u_1^\pm(\alpha).
\]
Substituting these expressions into the nonlinear interface condition
yields

\[
s = g(u_0^+(\alpha) + s u_1^+(\alpha),
      u_0^-(\alpha) + s u_1^-(\alpha)).
\]

If $g$ is polynomial of degree $k$, the right-hand side is polynomial in
$s$ of degree at most $k$, which yields the result.
\end{proof}
In particular, for the quadratic interface law
\[
[u]_\alpha = \lambda u^+(\alpha)u^-(\alpha),
\]
the reduced equation for $s$ is quadratic and may be solved
analytically or by a simple scalar Newton iteration.
Thus the nonlinear interface problem reduces to solving a scalar nonlinear equation for the interface variable $s$, while the remaining problem for $u_0$ is linear.
\rev{
\begin{remark}[Extension to nonzero flux jumps]
The scalar reduction extends directly to interface problems with
prescribed flux jumps.  Consider the interface conditions
\[
[\beta u_n]_\Gamma=q,
\qquad
[u]_\Gamma=g(u^+,u^-),
\]
where $q$ is prescribed.

In this case the continuous component $u_0$ is defined by
\[
[u_0]_\Gamma=0,
\qquad
[\beta \nabla u_0\cdot n]_\Gamma=q.
\]

The unit--jump response mode $u_1$ is defined by
\[
[u_1]_\Gamma=1,
\qquad
[\beta \nabla u_1\cdot n]_\Gamma=0.
\]

The representation
\[
u=u_0+s\,u_1
\]
remains valid, and substitution into the nonlinear interface
condition yields the same reduced scalar equation
\[
s=
g\!\left(
u_0^+ + s u_1^+,
u_0^- + s u_1^-
\right).
\]

Thus prescribed flux jumps affect only the construction of the
continuous component $u_0$, while the nonlinear interface coupling
remains entirely represented by the scalar variable $s$.
\end{remark}

}

\subsection{Explicit form of the unit–jump response}

In one dimension the function $u_1$ can be written explicitly. Since
$-(\beta u_{1,x})_x=0$ in each subdomain, the solution is linear on
$\Omega^-$ and $\Omega^+$.

Let
\[
u_1(x)=
\begin{cases}
a^- x + b^- , & x<\alpha,\\
a^+ x + b^+ , & x>\alpha.
\end{cases}
\]
The boundary conditions
\[
u_1(-1)=0, \qquad u_1(1)=0
\]
together with the interface conditions
\[
[u_1]_\alpha=1, \qquad [\beta u_{1,x}]_\alpha=0
\]
determine the coefficients uniquely.
Solving these relations yields
\[
a^-=\frac{\beta^+}{\beta^+(\alpha+1)+\beta^-(1-\alpha)},
\qquad
a^+=\frac{\beta^-}{\beta^+(\alpha+1)+\beta^-(1-\alpha)}.
\]
Thus the unit–jump response is piecewise linear and depends only on the
diffusion coefficients and the interface location.

Let $\mathcal{T}_h$ be a partition of $\Omega$ such that the interface point $\alpha$
is a mesh node.  We denote by $V_h$ the set of continuous piecewise
linear functions associated with $\mathcal{T}_h$ satisfying the boundary
conditions
\[
V_h = \{ v_h \in C(\Omega) : v_h|_K \text{ is linear on each element } K,
\; v_h(-1)=\xi,\; v_h(1)=\eta \}.
\]

The weak formulation of the elliptic interface problem is:
find $u \in H^1(\Omega^- \cup \Omega^+)$ such that
\begin{equation}
\int_\Omega \beta u_x v_x\,dx
=
\int_\Omega f v\,dx ,
\qquad \forall v \in H_0^1(\Omega),
\end{equation}
together with the nonlinear interface condition
\begin{equation}
[u]_\alpha = g(u^+(\alpha),u^-(\alpha)).
\end{equation}

\subsection{Discrete decomposition}

Following the continuous construction, we compute a discrete zero–jump solution
$u_{0,h}\in V_h$ satisfying
\begin{equation}
\int_\Omega \beta u_{0,h,x} v_{h,x}\,dx
=
\int_\Omega f v_h\,dx ,
\qquad \forall v_h \in V_{h,0},
\end{equation}
where
\[
V_{h,0} = \{ v_h\in V_h : v_h(-1)=0,\; v_h(1)=0 \}.
\]

Because the mesh is fitted to the interface, the function $u_{0,h}$ is
continuous at $\alpha$ and satisfies the flux continuity condition
automatically.

Next we define the discrete unit–jump mode $u_1$.
Since the interface coincides with a mesh node, $u_1$ can be obtained
explicitly as the piecewise linear function satisfying
\begin{equation}
-(\beta u_{1,x})_x = 0
\end{equation}
in each subdomain, together with
\[
u_1(-1)=0, \qquad u_1(1)=0,
\]
and
\[
[u_1]_\alpha = 1, \qquad [\beta u_{1,x}]_\alpha = 0 .
\]

The discrete solution is then represented as
\begin{equation}
u_h = u_{0,h} + s_h u_1 ,
\end{equation}
where $s_h$ is the discrete jump amplitude.

\subsection{Discrete scalar interface equation}

The traces of the numerical solution at the interface are
\begin{equation}
u_h^\pm(\alpha)
=
u_{0,h}^\pm(\alpha) + s_h u_1^\pm(\alpha).
\end{equation}

Substituting these expressions into the nonlinear interface condition yields a
scalar equation for $s_h$:
\begin{equation}
s_h =
g\!\left(
u_{0,h}^+(\alpha)+s_h u_1^+(\alpha),
u_{0,h}^-(\alpha)+s_h u_1^-(\alpha)
\right).
\end{equation}

This scalar equation can be solved by Newton iteration or by direct root
finding.  Once $s_h$ is obtained, the numerical solution is recovered from
\begin{equation}
u_h = u_{0,h} + s_h u_1 .
\end{equation}

Thus the nonlinear interface problem reduces to solving a single scalar
equation, while the finite element computation for $u_{0,h}$ remains linear and
standard.
\subsection*{Algorithm 1: Scalar interface reduction}

Given the interface point $\alpha$ and the nonlinear jump law
$[u]_\alpha = g(u^+,u^-)$, the finite element method proceeds as follows.

\begin{enumerate}
\item \textbf{Bulk solve.}
Solve the standard finite element problem for the continuous component
$u_{0,h}$.

\item \textbf{Interface response.}
Construct the unit--jump response function $u_1$ satisfying
\[
[u_1]_\alpha = 1, \qquad [\beta u_{1,x}]_\alpha = 0,
\]
with homogeneous boundary conditions.

\item \textbf{Scalar interface equation.}
Determine the interface variable $s_h$ from
\[
s_h =
g\big(u_{0,h}^+(\alpha)+s_h u_1^+(\alpha),
      u_{0,h}^-(\alpha)+s_h u_1^-(\alpha)\big).
\]

\item \textbf{Reconstruction.}
Recover the numerical solution
\[
u_h = u_{0,h} + s_h u_1 .
\]
\end{enumerate}

\begin{remark}[Interface response interpretation]

The function $u_1$ may be viewed as the response of the differential
operator to a unit jump imposed at the interface. In this sense,
$u_1$ plays a role analogous to a Green-type response associated with
an interface discontinuity rather than a point source.
\end{remark}

\subsection{Implications for numerical accuracy}

The scalar decomposition also provides a useful representation of the numerical
error and explains why the observed error profile is governed by the
unit--jump response.

\begin{theorem}[Error decomposition and interface response mode]\label{thm:error_mode}
Let the exact solution be written as
\[
u = u_0 + s u_1,
\]
and let the numerical approximation be reconstructed in the form
\[
u_h = u_{0,h} + s_h u_{1,h}.
\]
Then the error satisfies the exact decomposition
\begin{equation}\label{eq:error-splitting}
u-u_h
=
(u_0-u_{0,h})
+
s(u_1-u_{1,h})
+
(s-s_h)u_{1,h}.
\end{equation}
Consequently,
\begin{equation}\label{eq:error-bound}
\|u-u_h\|
\le
\|u_0-u_{0,h}\|
+
|s|\,\|u_1-u_{1,h}\|
+
|s-s_h|\,\|u_{1,h}\|.
\end{equation}

In particular, if the unit--jump response is represented exactly, that is,
if
\[
u_{1,h}=u_1,
\]
then
\begin{equation}\label{eq:error-leading}
u-u_h
=
(u_0-u_{0,h})
+
(s-s_h)u_1.
\end{equation}
Hence, after subtracting the standard bulk discretization error
$u_0-u_{0,h}$, the remaining error is exactly proportional to the
interface response function $u_1$.
\end{theorem}

\begin{proof}
Using the representations
\[
u=u_0+s u_1,
\qquad
u_h=u_{0,h}+s_h u_{1,h},
\]
we subtract the two expressions to obtain
\[
u-u_h
=
(u_0-u_{0,h})
+
s u_1 - s_h u_{1,h}.
\]
Adding and subtracting $s u_{1,h}$ gives
\[
s u_1 - s_h u_{1,h}
=
s(u_1-u_{1,h}) + (s-s_h)u_{1,h}.
\]
Substituting this into the previous identity yields
\eqref{eq:error-splitting}. The norm estimate \eqref{eq:error-bound}
follows immediately from the triangle inequality. If $u_{1,h}=u_1$,
the second term vanishes, and \eqref{eq:error-leading} follows.
\end{proof}

Theorem \ref{thm:error_mode} shows that, when the unit--jump response is represented
exactly, the interface component of the numerical error is completely
determined by the scalar jump error $s-s_h$. Figure~1 confirms this
relationship numerically.

\begin{figure}[t]
\centering
\includegraphics[width=0.8\textwidth]{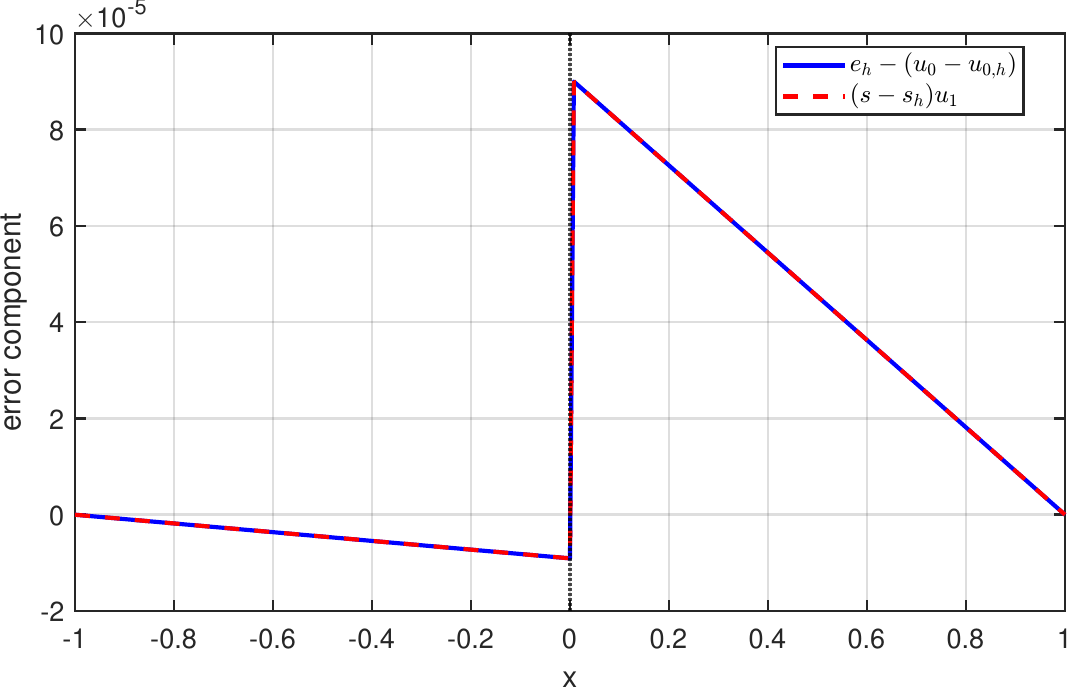}
\caption{Confirmation of Theorem~\ref{thm:error_mode}.
The curves $(u-u_h)-(u_0-u_{0,h})$ and $(s-s_h)u_1$
coincide numerically, confirming that the interface
error component is proportional to the unit--jump
response function $u_1$.}
\label{fig:interface_mode}
\end{figure}
\begin{remark}[Interface localization]
The decomposition
\[
u-uh=(u_0-u_{0,h})+(s-s_h)u_1
\]
also explains the localization of numerical errors near the
interface that is frequently observed in interface computations.
Since $u_1$ is generated by a unit jump at the interface,
the component $(s-s_h)u_1$ naturally attains its largest
magnitude near the interface.
\end{remark}
\rev{
\subsection*{Algebraic interpretation}

The Schur complement interpretation may be made more explicit at the
discrete level.

Suppose the discrete unknowns are partitioned into bulk variables
$U$ and the interface variable $s$.  The nonlinear discrete system
may be written schematically as
\[
A U + B s = F,
\]
together with a nonlinear interface relation
\[
s = G(U,s).
\]

Assuming $A$ is invertible, the bulk variables may be eliminated,
yielding
\[
U=A^{-1}(F-Bs).
\]

Substituting this expression into the interface relation gives
\[
s=\Phi(s),
\]
where
\[
\Phi(s)
=
G\!\left(A^{-1}(F-Bs),s\right).
\]

This reduced scalar equation is precisely the nonlinear equation
obtained from the interface reduction.  In this sense, the scalar
interface equation may be viewed as a nonlinear Schur complement
associated with the interface degree of freedom after elimination of
all bulk variables.
}
%\section{Extension to parabolic interface problems}

The scalar interface reduction extends naturally to time–dependent interface
problems.  We consider the parabolic equation
\begin{equation}
u_t - (\beta u_x)_x = f(x,t),
\qquad x\in\Omega^- \cup \Omega^+, \quad t>0,
\end{equation}
with boundary conditions
\begin{equation}
u(-1,t)=\xi(t), \qquad u(1,t)=\eta(t),
\end{equation}
and interface conditions
\begin{equation}
[u]_\alpha = g(u^+(\alpha,t),u^-(\alpha,t)),
\qquad
[\beta u_x]_\alpha = 0 .
\end{equation}

An initial condition
\begin{equation}
u(x,0) = u_0(x)
\end{equation}
is also prescribed.

\subsection{Time discretization}

Let $t^n = n\Delta t$ denote the time levels.  Using backward Euler
time discretization we obtain the semi–discrete problem
\begin{equation}
\frac{u^{n+1}-u^n}{\Delta t}
- (\beta u^{n+1}_x)_x
=
f^{n+1}.
\end{equation}

\subsection{Scalar interface reduction}

At each time step the solution is represented in the form
\begin{equation}
u^{n+1} = u_0^{n+1} + s^{\,n+1} u_1 ,
\end{equation}
where $u_1$ is the same unit–jump response mode introduced in the elliptic
problem.

The function $u_0^{n+1}$ is obtained by solving the linear finite element
problem
\begin{equation}
\int_\Omega \frac{u_0^{n+1}-u^n}{\Delta t} v_h \, dx
+
\int_\Omega \beta u_{0,x}^{n+1} v_{h,x} \, dx
=
\int_\Omega f^{n+1} v_h \, dx ,
\qquad \forall v_h \in V_{h,0}.
\end{equation}

The interface traces of the numerical solution are then
\begin{equation}
u^{n+1,\pm}(\alpha)
=
u_0^{n+1,\pm}(\alpha) + s^{\,n+1} u_1^\pm(\alpha).
\end{equation}

Substituting these expressions into the nonlinear interface condition yields
the scalar equation
\begin{equation}
s^{\,n+1}
=
g\!\left(
u_0^{n+1,+}(\alpha) + s^{\,n+1} u_1^+(\alpha),
u_0^{n+1,-}(\alpha) + s^{\,n+1} u_1^-(\alpha)
\right).
\end{equation}

Thus each time step requires solving a linear finite element problem for
$u_0^{n+1}$ followed by a scalar nonlinear equation for the interface
variable $s^{\,n+1}$.  The numerical solution is then reconstructed as
\begin{equation}
u^{n+1} = u_0^{n+1} + s^{\,n+1} u_1 .
\end{equation}
%\section{Numerical experiments}

The scalar reduction developed in Section 2 leads to a computational
procedure in which the bulk components are obtained from linear problems,
while the nonlinear interface condition is resolved through a scalar
equation for the interface variable $s$. The following experiments are
designed to verify the accuracy of this reduction and its impact on the
interface quantities.
In this section we verify the method on benchmark problems of the same type as those used in \cite{Chou2025}. The numerical evidence is organized around the interface quantities, namely the left and right traces and, in the nonlinear case, the jump variable.

\subsection{Elliptic problem with prescribed jump}
We first consider
\[
-(\beta u')' = 0.1\sin(\pi x), \qquad -1\le x\le 1,
\]
with boundary conditions
\[
u(-1)=u(1)=0,
\]
interface point $\alpha=0$, coefficients
\[
\beta^- = 1,
\qquad
\beta^+ = 0.1,
\]
and prescribed jump
\[
[u]_\alpha = 1.1.
\]
The exact solution is
\[
u(x)=
\begin{cases}
\dfrac{0.1}{\pi^2}\sin(\pi x)-0.1(x+1), & -1\le x\le 0,\\
\\
\dfrac{1}{\pi^2}\sin(\pi x)-(x-1), & 0<x\le 1.
\end{cases}
\]
Since the jump is prescribed exactly, the numerical method reproduces the
interface jump to machine precision. The relevant quantities are therefore
the interface trace errors and the global $L^\infty$ error. Table~1 reports
these errors for a sequence of uniform meshes, where the column labeled $MR$
denotes the mesh resolution, i.e. the number of subintervals in the uniform
partition of $\Omega$. The results show that the interface traces are computed
to machine precision and that the $L^\infty$ error converges with second-order
accuracy.

\begin{table}[htbp]
\centering
\caption{Elliptic linear-jump example: interface trace and $L^\infty$ errors.}
\begin{tabular}{c c c c c}
\toprule
MR & $h$ & left trace error & right trace error & $L^\infty$ error \\
\midrule
8   & 0.125000   & $1.39e{-17}$ & $0.00e+{00}$   & $6.5989e{-4}$ \\
16  & 0.062500   & $9.71e{-17}$ & $1.11e{-16}$ & $1.6331e{-4}$ \\
32  & 0.031250   & $1.53e{-16}$ & $2.22e{-16}$ & $4.0724e{-5}$ \\
64  & 0.015625   & $3.75e{-16}$ & $4.44e{-16}$ & $1.0175e{-5}$ \\
128 & 0.0078125  & $5.13e{-16}$ & $4.44e{-16}$ & $2.5433e{-6}$ \\
256 & 0.00390625 & $7.33e{-15}$ & $7.33e{-15}$ & $6.3579e{-7}$ \\
\bottomrule
\end{tabular}
\end{table}

In this benchmark the scalar jump variable is therefore recovered exactly,
so the interface contribution $(s-s_h)u_1$ in Theorem~\ref{thm:error_mode}
is negligible and the numerical error is dominated by the bulk discretization
term $u_0-u_{0,h}$. Figure~2 illustrates the resulting spatial error profile
for the linear prescribed-jump case. Since the interface jump is imposed exactly,
the figure reflects primarily the bulk discretization error rather than
the pure interface response mode.

\begin{figure}[h]
\centering
\includegraphics[width=0.65\textwidth]{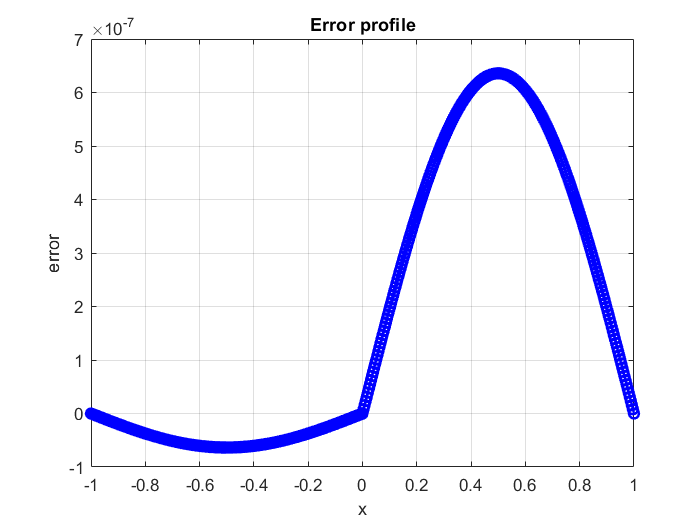}
\caption{Error profile for the elliptic interface test.}
\end{figure}

\subsection{Parabolic problem with prescribed jump}
Next we consider the linear-jump parabolic benchmark. Backward Euler is used in time with $\Delta t = h^2$. Since the jump is prescribed exactly, the meaningful interface diagnostics are the left and right trace errors at the final time $T=1$.

\begin{align}
    &u_{t}-(\beta u')' = f(x,t), \quad -1\leq x \leq 1, t\in (0,T]\\
        &u(-1,t)=0, \quad u(1,t)=0,\\
        &[u]_{\alpha}=\mu,[\beta u']_{\alpha}=0,\\
    &u(x,0)=u_{0}(x)
\end{align}
with $\alpha=0$, $\mu=1$, $\beta^-=1, \beta^+=0.1$.
The exact solution
\begin{equation}
    u(x,t)=\begin{cases}
    e^{t}\sin(\pi x)+A(x+1) & -1\le x\le 0,\\
     10e^{t}\sin(\pi x)+B(x-1)& 0< x\le 1,
\end{cases}
\end{equation}
where
\begin{equation}
 A = -\dfrac{\mu}{11}, \quad  B=10A.
\end{equation}

\begin{table}[htbp]
\centering
\caption{Parabolic linear-jump example: final-time interface trace and $L^\infty$ errors.}
\begin{tabular}{c c c c c c}
\toprule
MR & $h$ & $\Delta t$ & left trace error & right trace error & $L^\infty(T)$ \\
\midrule
8   & 0.125000  & 0.015625   & $2.0999e{-3}$ & $2.0999e{-3}$ & $1.1686e{-2}$ \\
16  & 0.062500  & 0.00390625 & $5.2540e{-4}$ & $5.2540e{-4}$ & $2.8613e{-3}$ \\
32  & 0.031250  & 0.00097656 & $1.3138e{-4}$ & $1.3138e{-4}$ & $7.1180e{-4}$ \\
64  & 0.015625  & 0.00024414 & $3.2845e{-5}$ & $3.2845e{-5}$ & $1.7792e{-4}$ \\
128 & 0.0078125 & 0.00006104 & $8.2115e{-6}$ & $8.2115e{-6}$ & $4.4466e{-5}$ \\
\bottomrule
\end{tabular}
\end{table}

The results in Table 2 confirm second-order spatial accuracy for the interface traces in the parabolic linear-jump case.

\subsection{Parabolic problem with nonlinear jump}\label{EXAMPLE 3}
Finally we test the nonlinear jump law
\[
[u]_\alpha = \lambda u^+(\alpha)u^-(\alpha).
\]
At each time step, the scalar interface variable $s$ is obtained from the reduced nonlinear equation

Consider
\begin{align*}
    &u_{t}-(\beta u')' = f, \quad -1\leq x \leq 1, t\in(0,T]\\
    &u(-1,t)=0, \quad u(1,t)=2,\\
    &[\beta u']_{\alpha}=0,\\
    &[u]_{\alpha}=0.5 u^{+}u^{-},\\
    &u(x,0)=u_0(x),
\end{align*}
where $\beta^-=1$ and $ \beta^+=0.1$,

\begin{equation}
    u(x,t) =
    \begin{cases}
       e^{-t} \sin(\pi x) + x+1 & -1\le x\le 0,\\
         2+10(\pi e^{- t}+1)\big(x-x^2\big) & 0 < x\le 1.
    \end{cases}
\end{equation}
For the benchmark used here, it can be checked analytically that the exact jump is constant. Thus one can directly measure the jump error at the final time $T=2$ together with the left and right interface trace errors.

\begin{table}[htbp]
\centering
\caption{Parabolic nonlinear-jump example: final-time jump, interface trace, and $L^\infty$ errors.}
\begin{tabular}{c c c c c c c}
\toprule
MR & $h$ & $\Delta t$ & jump error & left trace error & right trace error & $L^\infty(T)$ \\
\midrule
8   & 0.125000  & 0.015625   & $2.5414e{-2}$ & $8.5684e{-3}$ & $3.3982e{-2}$ & $3.0749e{-2}$ \\
16  & 0.062500  & 0.00390625 & $6.3784e{-3}$ & $2.1322e{-3}$ & $8.5106e{-3}$ & $8.1558e{-3}$ \\
32  & 0.031250  & 0.00097656 & $1.5962e{-3}$ & $5.3243e{-4}$ & $2.1286e{-3}$ & $2.0879e{-3}$ \\
64  & 0.015625  & 0.00024414 & $3.9914e{-4}$ & $1.3307e{-4}$ & $5.3221e{-4}$ & $5.2737e{-4}$ \\
128 & 0.0078125 & 0.00006104 & $9.9791e{-5}$ & $3.3265e{-5}$ & $1.3306e{-4}$ & $1.3247e{-4}$ \\
\bottomrule
\end{tabular}
\end{table}

From Table 3, we see that the jump error, left trace error, and right trace error all converge with order two. This supports the scalar interface reduction developed in the preceding sections.

These results confirm that the nonlinear interface problem is effectively
reduced to the computation of a scalar variable $s$, while the bulk solution
retains the accuracy of a standard linear discretization.
\rev{
\subsection*{Computational efficiency}

An important advantage of the proposed formulation is that the
nonlinear interface problem is reduced to a scalar nonlinear equation.

Traditional nonlinear finite element formulations typically require
repeated assembly and solution of large nonlinear algebraic systems.
In contrast, the present approach consists of

\begin{enumerate}
\item computation of the continuous component $u_0$,
\item construction of the unit--jump response mode $u_1$,
\item solution of a scalar nonlinear equation for the jump variable
$s$.
\end{enumerate}

The dimension of the nonlinear problem is therefore independent of
the mesh size.  As the mesh is refined, the additional nonlinear
computational cost remains essentially constant.
}
\section{Conclusion}

We have presented a simple scalar interface reduction for finite element discretizations of nonlinear interface problems. The central
idea is to represent the solution as the sum of a continuous component and a
unit–jump response mode.  This decomposition isolates the interface
discontinuity into a single scalar variable while the bulk finite element
problem remains linear and standard. The results confirm that the proposed formulation provides a simple and robust framework for treating nonlinear interface conditions within standard finite element discretizations.

The nonlinear interface condition is thereby reduced to a scalar nonlinear
equation for the jump amplitude.  This formulation provides a simple and
computationally efficient treatment of nonlinear interface laws.  The same
structure extends naturally to time–dependent problems, where each time step
requires solving one linear finite element system followed by a scalar
interface equation.

The decomposition also admits a natural structural interpretation.  The
unit–jump mode may be viewed as the response of the differential operator to
an interface source, while the scalar interface equation plays the role of a
reduced Schur complement relation governing the interface variable.

Numerical experiments for several elliptic and parabolic benchmark problems
demonstrate second–order accuracy for the interface traces and for the computed
jump variable.  The results confirm that the proposed formulation provides a
simple and robust framework for treating nonlinear interface conditions within
standard finite element discretizations.

The scalar reduction is not restricted to one dimension. In higher
dimensions, the decomposition $u = u_0 + s u_1$ remains valid, where
$u_1$ represents the response of the differential operator to a unit
interface jump and depends only on the interface geometry and the
coefficients. The nonlinear interface condition then reduces to a
low-dimensional nonlinear system associated with interface degrees of
freedom. In the case of a single connected interface component, this
system remains scalar, and in general it involves one degree of freedom
per interface component. This perspective suggests a natural extension
of the present formulation to multidimensional interface problems. The essential structure of the reduction, namely the separation of bulk
and interface variables, is therefore preserved in higher dimensions.

\bibliographystyle{siam}
\bibliography{references}

@book{li2006immersed,
  title={The immersed interface method: numerical solutions of PDEs involving interfaces and irregular domains},
  author={Li, Zhilin and Ito, Kazufumi},
  year={2006},
  publisher={SIAM}
}

@article{Chou2025,
  title={A finite difference method for an interface problem with a nonlinear jump condition},
  author={Chou, So-Hsiang and Khaemba, Caleb and Wachira, Alice},
  journal={Int. J. Numer. Anal. Model.},
  volume={22},
  number={6},
  pages={801-823},
  year={2025},
  publisher={Global Science Press}
}

@article{fries2010extended,
  title={The extended/generalized finite element method: an overview of the method and its applications},
  author={Fries, Thomas-Peter and Belytschko, Ted},
  journal={International journal for numerical methods in engineering},
  volume={84},
  number={3},
  pages={253--304},
  year={2010},
  publisher={Wiley Online Library}
}

@article{jo2018enriched,
  title={Enriched $P_1$-Conforming Methods for Elliptic Interface Problems with Implicit Jump Conditions},
  author={Jo, Gwanghyun and Kwak, Do Y.},
  journal={Advances in Mathematical Physics},
  volume={2018},
  pages={Article ID 9891281},
  year={2018},
  publisher={Hindawi},
  doi={10.1155/2018/9891281}
}

@article{babuvska1997partition,
  title={The partition of unity method},
  author={Babu{\v{s}}ka, Ivo and Melenk, Jens M},
  journal={International journal for numerical methods in engineering},
  volume={40},
  number={4},
  pages={727--758},
  year={1997},
  publisher={Wiley Online Library}
}

@Manual{R,
    title = {R: A Language and Environment for Statistical Computing},
    author = {{R Core Team}},
    organization = {R Foundation for Statistical Computing},
    address = {Vienna, Austria},
    year = {2014},
    url = {http://www.R-project.org/},
  }

@article{HouWu1997,
  author  = {Hou, Thomas Y. and Wu, Xiao-Hui},
  title   = {A multiscale finite element method for elliptic problems in composite materials and porous media},
  journal = {Journal of Computational Physics},
  volume  = {134},
  number  = {1},
  pages   = {169--189},
  year    = {1997}
}

@article{lin2015,
  author  = {Tao Lin and Dongwoo Sheen and Xiaohui Zhang},
  title   = {A locking-free immersed finite element method for planar elasticity interface problems},
  journal = {Journal of Computational Physics},
  volume  = {276},
  pages   = {205--222},
  year    = {2014}
}

@article{leveque1994immersed,
  title={The immersed interface method for elliptic equations with discontinuous coefficients and singular sources},
  author={LeVeque, Randall J and Li, Zhilin},
  journal={SIAM Journal on Numerical Analysis},
  volume={31},
  number={4},
  pages={1019--1044},
  year={1994},
  publisher={SIAM}
}

@article{jeon2022immersed,
  author    = {Jeon, Youngmok and Yi, Son-Young},
  title     = {The Immersed Interface Hybridized Difference Method for Parabolic Interface Problems},
  journal   = {Numerical Mathematics: Theory, Methods and Applications},
  volume    = {15},
  number    = {2},
  pages     = {367--393},
  year      = {2022},
  publisher = {Global Science Press},
  doi       = {10.4208/nmtma.2022.0003}
}

@article{Chou2026FluxRecovery,
  author  = {Champike Attanayake and So-Hsiang Chou},
  title   = {Interface Reduction for Elliptic Interface Problems with Conservative Flux Reconstruction},
  journal = {arXiv preprint arXiv:2605.09800},
  year    = {2026}
}

@article{chou2026lifting,
  title={A Lifting-Based Interface Reduction Framework for Nonlinear Transmission and Eigenvalue Problems},
  author={Chou, So-Hsiang},
  journal={arXiv preprint arXiv:2606.00468},
  year={2026}
}
\end{document}